\documentclass{article}
\usepackage{amsthm,amsmath,amsfonts,amssymb}

\title{Relative entropy and contraction for  extremal shocks of Conservation Laws up to a shift}

\author{Alexis F. Vasseur \\ The University of Texas at Austin}
\newlength{\hchng}
\newlength{\vchng}
\newtheorem{theo}{Theorem}[section]
\newtheorem{prop}{Proposition}

\newtheorem{lemm}{Lemma}
\newtheorem{defi}{Definition}

\newcommand{\ds}{\displaystyle}
\newcommand{\R}{\mathbb R}
\newcommand{\eps}{\varepsilon}

\newcommand{\dt}{\partial_t}
\newcommand{\dx}{\partial_x}

\newcommand{\mV}{\mathcal{V}}

\newcommand{\mU}{\mathcal{U}}
\newcommand{\Vs}{V_{\mathrm{max}}}
\newcommand{\Vi}{V_{\mathrm{min}}}
\newcommand{\sU}{S}
\newcommand{\E}{\mathcal{E}}
\newcommand{\Oa}{\mathcal{O}_a}
\newcommand{\midd}{\thinspace\vert\thinspace}
\DeclareMathOperator*{\esssup}{ess\,sup}

\begin{document}
\maketitle
\bibliographystyle{plain}
\noindent{\bf Abstract:}  
We consider systems of conservation laws endowed with a convex entropy.
We show the contraction, up to a translation, to extremal entropic shocks, for a pseudo-distance based on the notion of relative entropy.
The contraction holds for bounded entropic weak solutions having an additional trace property. 
The pseudo-distance depends only on the fixed extremal entropic shocks in play.
 In particular, it can be chosen uniformly for any  entropic weak solutions which are compared to a fixed shock. The boundedness of the solutions  controls the strength of the shift needed to get the contraction.   However,  no $BV$ estimate is needed on the weak solutions considered. The theory holds without smallness condition.  For fluid mechanics, the theory handles solutions with vacuum.   



\section{Introduction}
In this article, we provide new developments on the theory,  introduced in \cite{LV}, of $L^2$ stability of extremal shocks for systems of conservation laws, based on the relative entropy. 
\vskip0.5cm
 The theory of Kruzhkov shows that  solutions to scalar conservation laws provide a contraction in $L^1$. This result is not true for the $L^2$ norm, even when considering the distance to a fixed shock. Leger showed in \cite{Leger} that, however,  the contraction to a fixed shock is true up to a drift. Consider a scalar conservation law
 $$
 \dt u+\dx f(u)=0
 $$
 with a strictly convex flux $f$, and a fixed entropic shock  $(U_L,U_R, \sigma)$. We recall that a shock is a special entropic solution
  $S(t,x)$ equal to $U_L$ for $x<\sigma t$ and equal to $U_R$ for $x>\sigma t$. 
   Leger showed that for every entropic bounded solution $u$ of the scalar conservation law, there exists a Lipschitz shift
 $t\to X(t)$ such that 
 $$
 \int_{\R}|u(t,x)-S(t,x-X(t))|^2\,dx
 $$
 is  non increasing in time. Note that this property is valid for any relative entropy pseudo-distance defined from a strictly convex entropy (see section \ref{section_presentation} for the definition of the relative entropy). 
 \vskip0.5cm
 A thorough study of this kind of contraction property for system of conservation laws can be found in 
 \cite{SV}. Although the  contraction does not hold for most  systems, it has been showed  in \cite{LV} that the relative entropy can be used to show a result of strong stability of extremal shocks and contact discontinuities, for 
 the $L^2$ distance, up to a shift.   
 \vskip0.5cm
In this article, we focus on extremal shocks verifying the Liu condition. For any such shock, we construct a pseudo-distance, still based on the relative entropy, but not anymore homogenous in $x$. This pseudo-distance  induces a contraction, up to a drift, to this shock. Moreover, we show that the pseudo-distance  does not depend on any quantitative property of the weak solutions (not even their $L^\infty$ norm). Only the drift does. 
\vskip0.5cm
The contraction is shown for a wide class of systems of conservation laws endowed with a convex entropy, and for bounded weak entropic solutions verifying the following trace property (see \cite{LV}).
\begin{defi}\label{defi_trace}
Let $U\in L^\infty(\R^+\times\R)$. We say that $U$ verifies the strong trace property if for any Lipschitzian curve $t\to X(t)$, there exists two bounded functions $U_-,U_+\in L^\infty(\R^+)$ such that for any $T>0$
$$
\lim_{n\to\infty}\int_0^T\sup_{y\in(0,1/n)}|U(t,x(t)+y)-U_+(t)|\,dt=\lim_{n\to\infty}\int_0^T\sup_{y\in(-1/n,0)}|U(t,x(t)+y)-U_-(t)|\,dt=0.
$$
\end{defi}
%\begin{defi}\label{defi_trace}
%Let $U\in L^\infty(\R^+\times\R)$. We say that $U$ verifies the regulated trace property if for any %Lipschitzian curve $t\to X(t)$, there exists two bounded functions $U_-,U_+\in L^\infty(\R^+)$ %such that for almost every $t>0$%any $T>0$
%%$$
%%\lim_{n\to\infty}\int_0^T\esssup_{y\in(0,1/n)}|U(t,x(t)+y)-U_+(t)|\,dt=\lim_{n\to\infty}\int_0^T%\esssup_{y\in(-1/n,0)}|%U(t,x(t)+y)-U_-(t)|\,dt=0.
%%$$
%\begin{align}
%\esslim_{y \to X(t)^+}|U(t, y)-U_+(t)| = \esslim_{y \to X(t)^-}|U(t, y)-U_-(t)| =0.
%\end{align}
%\end{defi}
Obviously, any $BV$ function verifies this strong trace property.
But this requirement is weaker than the $BV$ property.
% (cf. \cite{Dieudonne}).
Let us emphasize that this notion of trace is more restrictive than the strong trace introduced in \cite{Vasseur_trace}, which is 
 known to be verified for bounded solutions of scalar conservation laws. 
%But this requirement is weaker than the $BV$ property. Let us emphasize that this notion of trace is more restrictive %than the one introduced in \cite{Vasseur_trace} (where the $\sup$ is before the integral in time). 
%The usual strong trace property is known to be verified for bounded solutions of scalar conservation laws. 
%But this requirement is far weaker than the $BV$ property. The strong trace property is known to be verified for bounded solutions of scalar conservation laws.
This has been shown in the multidimensional case, first with a non-degeneracy property, in \cite{Vasseur_trace}. In the one-dimensional case, a different proof based on compensated compactness was proposed by  Chen and Rascle \cite{Chen_trace}.
For a general flux function the strong trace problem has been solved in the 1D case in \cite{KV}. The general multidimensional case has been obtained by Panov \cite{Panov,Panov2} (see also Kwon \cite{Kwon},  De Lellis,  Otto,  and Westdickenberg \cite{DeLellis} for interesting generalizations). In the case of systems, this is mainly an open problem. This has been shown only for the particular case of isentropic gas dynamics with $\gamma=3$ for traces in time (traces in space can be shown the same way) in \cite{Vasseur_gamma3}. 
%Unfortunately, there is, so far, no such result for the regulated trace property of definition \ref{defi_trace} 
%outside of the $BV$ theory.
Unfortunately, there are no such results for the strong  trace property of Definition \ref{defi_trace} outside of the usual $BV$ theory.
\vskip0.3cm
Stability of shocks in the class of $BV$ solutions has been investigated by a number of authors. In the case of small perturbations in $L^\infty\cap BV$, Bressan, Crasta, and Piccoli \cite{Bressan1} developed a powerful theory of $L^1$ stability for entropy solutions obtained by either the Glimm scheme \cite{Glimm} or the wave front-tracking method.  A simplified approach has been proposed by Bressan, Liu, and Yang \cite{Bressan2} and Liu and Yang \cite{liu}. (See also Bressan \cite{Bressan}.) The theory also works in some cases for small perturbations in $L^\infty\cap BV$ of large shocks. See, for instance, Lewicka and Trivisa \cite{Lewicka} or Bressan and Colombo \cite{Bressan3}.  
 \vskip0.3cm
However, our contraction result goes beyond the known results valid  in the class of $BV$ solutions, with perturbations small in $BV$. 
Our approach is based on the relative entropy method first used by Dafermos and DiPerna to show $L^2$ stability and uniqueness of Lipschitzian solutions to conservation laws
\cite{Dafermos4,Dafermos1,DiPerna}. Note that in \cite{DiPerna}, uniqueness of small shocks for strictly hyperbolic $2\times 2$ systems is shown in a class of admissible weak solutions with small oscillation in $L^\infty \cap BV$. The analysis in \cite{DiPerna} also implies the uniqueness of shocks for $2 \times 2$ systems in the Smoller-Johnson class \cite{Smoller}. In each case genuine nonlinearity is assumed. The ideas of DiPerna were developed further by Chen and Frid in the papers \cite{Frid2,Frid1}. In subsequent work, they established, together with Li \cite{Chen1}, 
the uniqueness of solutions to the Riemann problem in a large class of entropy solutions (locally $BV$ without smallness conditions) for the $3\times 3$ Euler system in Lagrangian coordinates. They also establish a large-time stability result in this context. See also Chen and Li \cite{Chen_Li} for an extension to the relativistic Euler equations. However, no stability in $L^2$ for all time is included in those results. 
\vskip0.3cm 
Our approach is based on fairly mild assumptions on the system and the Rankine-Hugoniot discontinuity (see \cite{LV}).
Basically, we need the discontinuity to be extremal (1-shock or n-shock and well separated from the other Hugoniot discontinuities), and verify the Liu condition.
We need also a property of growth of the strength of the shock along the shock curve, where the strength is measured via the relative entropy. Texier and Zumbrun showed in \cite{TZ} that the Lopatinski conditions are verified under our hypotheses. Barker, Freist\"uhler and Zumbrun constructed in \cite{BFZ}, for the Euler equation, some pressure laws for which the stability does not hold. Those pressures verify most of our assumptions. The only difference is that the increase of the strength  of the shock, along the shock curve, is measured by the entropy instead of the relative entropy.

Very little constraint is needed on the other shock families. Lax properties are typically enough. But we may even relax it
to cases where the system is neither genuinely nonlinear nor strictly hyperbolic, and even to cases where the shock curves are not well-defined. The theory works fine even for large shocks.  
\vskip0.5cm
The relative entropy method is also an important tool in the study of asymptotic limits to conservation laws. Applications of the relative entropy method in this context began with the work of Yau \cite{Yau} and have been studied by many others. For incompressible limits, see Bardos, Golse, Levermore \cite{Bardos_Levermore_Golse1,Bardos_Levermore_Golse2}, Lions and  Masmoudi \cite{Lions_Masmoudi}, Saint Raymond et al. \cite{SaintRaymond1,SaintRaymond2,SaintRaymond3,SaintRaymond4}. For the compressible limit, see Tzavaras \cite{Tzavaras_theory} in the context of relaxation 
and \cite{BV,BTV,MV} in the context of hydrodynamical limits. In all those papers, the method works as long as the limit solution is Lipschitz.  This is because the method is based on the strong stability of such solutions to the limit system of conservation laws: initial $ \eps$ perturbations lead to $C_T\eps$ perturbations at time $t<T$.  Roughly speaking, the convergence in driven by the stability of the limit function. This paper is part of a general program initiated in \cite{Vasseur_Book}  to apply this kind of method to shocks. 
It is well known that shocks are  not strongly $L^2$ stable as above (an initial $ \eps$ perturbation can lead to a $\sqrt{ \eps}$ perturbation in finite time).  However, this paper shows that extremal shocks are strongly $L^2$ stable, up to a suitable drift.   A  first application of the method  to the study of asymptotic limit to a shock can be found in \cite{KV} in the scalar case. 
\vskip0.3cm
The key idea of the proof is to find the proper way to construct the shift on the fly which induces the contraction. A similar construction was performed in 
 \cite{Vasseur_shock},  for the study of a semi-discrete shock for an isentropic gas with $\gamma=3$.
\vskip0.5cm

We will give a precise description of our hypotheses and main results in the next section. First let us mention a few particular cases in which our theory applies. Our first examples include the isentropic Euler system and the full Euler system for a polytropic gas. Both systems are treated in Eulerian coordinates. The isentropic Euler system is the following.
\begin{equation}\label{isentropic}
\left\{\begin{array}{l}
\ds{\dt \rho+\dx (\rho u)=0}\\[0.2cm]
\ds{\dt (\rho u )+\dx (\rho u^2+ P(\rho))=0}.
\end{array}\right.
\end{equation}
We assume a smooth pressure law $P: \R^+ \to \R$ with the following properties
\begin{align}\label{isentropic2}
P^{\prime}(\rho)>0, \qquad \quad [\rho P(\rho)]^{\prime \prime} \geq 0.
\end{align}
As usual, we consider only entropic solutions of this system, namely, those verifying additionally the entropy inequality:
$$
\dt \eta(\rho,\rho u)+\dx G(\rho, \rho u) \leq 0,
$$
with
 $$
 \eta(\rho, \rho u)=\frac{(\rho u)^2}{2 \rho} + S(\rho), \qquad G(\rho, \rho u) = \frac{(\rho u)^3}{2 \rho^2} + \rho u \thinspace S^{\prime}(\rho),
 $$
and with $S^{\prime \prime}(\rho) = \rho^{-1} P^{\prime}(\rho) > 0$. Note that we need only a single convex entropy, even if in this case there exists an entire family of convex entropies. 
The relative entropy defining a pseudo norm is given by
$$
\eta(\rho,\rho u| \bar{\rho}\bar u)=\rho \frac{|u-\bar u|^2}{2}+S(\rho|\bar\rho)\geq 0,
$$
where $S(\rho|\bar\rho)=S(\rho)-S(\bar\rho)-S'(\bar\rho)(\rho-\bar\rho)$.
For $P(\rho)=\rho^2$ (the shallow water equation), we have $S(\rho|\bar\rho)=|\rho-\bar\rho|^2$.
%For $U=(\rho,u)$ and $\hU=(\hat{\rho},\hat{u})$ we consider the relative entropy function (see later for the general formula)
%$$
%\eta(U/\hU)=\rho(u-\hat{u})^2/2+h(\rho/\hat{\rho}),\qquad h(\rho/\hat{\rho})=\frac{1}{\gamma-1}(\rho^\gamma-\hat{\rho}^\gamma-\gamma\hat{\rho}^%{\gamma-1}(\rho-\hat{\rho})).
%$$
%Note that $h(\rho,\hat{\rho})\geq 0$, and is equal to 0 if and only if $\rho=\hat{\rho}$.
%Hence we have $\eta(U/\hU)\geq0$ and is equal to 0 of and only if $\rho=\hat{\rho}$ and $\rho u=\hat{\rho}\hat{u}$.
\vskip0.3cm

The full Euler system reads
\begin{equation}\label{Euler}
\left\{\begin{array}{l}
\ds{\dt \rho+\dx (\rho u)=0}\\[0.3cm]
\ds{\dt (\rho u )+\dx (\rho u^2+ P)=0}\\[0.3cm]
\ds{\dt (\rho E )+\dx (\rho u E + uP)=0,}
\end{array}\right.
\end{equation}
where $E = \frac{1}{2}u^2 + e$. We describe here the case of a polytropic gas. Such a gas verifies the hypotheses stated in the next section which ensure the contraction property (see also \cite{LV}). Note that for more general cases, the hypotheses can be not verified. Counterexamples to the stability have been provided in such cases  by Barker, Freist\"uhler and Zumbrun in \cite{BFZ}. 
The equation of state for a polytropic gas is given by 
\begin{equation}\label{polytropic}
P = (\gamma-1) \rho e
\end{equation}
where $\gamma >1$. 
In that case, we consider the entropy/entropy-flux pair
\begin{equation}\label{Eulerconvex}
\eta(\rho, \rho u, \rho E) = (\gamma -1) \rho \ln \rho - \rho \ln e, \qquad G(\rho, \rho u, \rho E) = (\gamma -1) \rho u \ln \rho - \rho u \ln e,
\end{equation}
where, in conservative variables, we have $e = \ds{\frac{\rho E}{\rho} - \frac{(\rho u)^2}{2 \rho^2}}$. \\
In this case, the relative entropy is
$$
 \eta(\rho, \rho u, \rho E)= \rho \frac{|u-\bar u|^2}{2\bar\theta}+(\gamma-1)\rho\phi(\bar\rho|\rho)+\rho\phi(e|\bar e),
$$
where $\phi$ is the relative function associated to $\ln(1/x)$, $\phi(x|y)=\ln(y/x)+(1/y)(x-y)\geq0$.

For the Euler systems (\ref{isentropic}) and (\ref{Euler}), we have the following theorem.

\begin{theo}\label{theoEuler}
Consider a  shock $(U_L,U_R)=((\rho_L,u_L),(\rho_R,u_R))$ with velocity $\sigma$ associated  to the system (\ref{isentropic})-(\ref{isentropic2}), (resp. $(U_L,U_R)=((\rho_L,u_L,E_L),(\rho_R,u_R,E_R))$ associated  to the system (\ref{Euler})-(\ref{polytropic})). Then there exists a constant $a>0$ depending only on the shock with the following property.
\vskip0.3cm
Consider  any $K>0$.  There exists $C_K>0$ such that,
for any weak entropic solution $U=(\rho,u)\in L^\infty(\R^+ \times\R)$ of (\ref{isentropic}) (resp.  $U=(\rho, u,E)\in L^\infty(\R^+ \times \R)$ of (\ref{Euler})) 
verifying the strong trace property of Definition \ref{defi_trace} and such that $\|(\rho, u)\|_{L^\infty}\leq K$, (resp. $\|(\rho, u,E)\|_{L^\infty}\leq K$), 
 there exists a Lipschtiz path $x(t)$ such that for any $t>0$, the pseudo norm
$$
\int_{-\infty}^0\eta(U(t,x+x(t))| U_L)\,dx+a\int_0^\infty \eta(U(t,x+x(t)|U_R)\,dx
$$
is  non increasing in time. Moreover, for every $t>0$:
\begin{eqnarray*}
&&|x'(t)|\leq C_K,\\
&&|x(t)-\sigma t|\leq C_K\sqrt{t}\|U_0-S\|_{L^2(\R)},
\end{eqnarray*}
where $U_0=U(t=0)$, and $S(x)=U_L$ for $x<0$ and $S(x)=U_R$ for $x>0$.
\vskip0.3cm
Especially, we have for every $t>0$
$$
\|U(t,\cdot+x(t))- S\|_{L^2(\R)}\leq C_K \|U_0- S\|_{L^2(\R)}.$$
\vskip0.3cm
Finally, we can choose $a<1$ of a 1-shock, and $a>1$ for a n-shock (n=2 for the isentropic case, and n=3 for the full Euler).
\end{theo}
Note that the pseudo distance (as $a$) depends only on the considered shock $(U_L,U_R,\sigma)$.
Especially, it does not depend on any quantitative property of the weak solution $U$ (not even its $L^\infty$ norm). It shows that the profile of the shock (up to a drift) is extremely stable with respect to large perturbations. The drift needed is more sensitive to the perturbation. Especially, its strength depends on the $L^\infty$ norm of the perturbation. We will show that the relative entropy is equivalent to the $L^2$ norm on any bounded sets for $(\rho,u)$ (respectively $(\rho,u,E)$). Unfortunately, it is not true globally. This explains why the result for the $L^2$ norms depends on the $L^\infty$ norm of $U$.
\vskip0.3cm
A good feature of the theory is that it can handle vacuum. Indeed, the solutions are not assumed to be away from vacuum. This gives a $L^2$ stability results up to the translation $x(t)$.  For $2 \times 2$ systems, all Rankine-Hugoniot discontinuities are extremal. Hence Theorem \ref{theoEuler} applies to all admissible shocks of (\ref{isentropic}). For the full Euler system, 
all shocks are 1-shocks or n-shocks (3-shocks in this case), so again the result applies to any entropy admissible shock. However, our result does not provide the stability of contact discontinuities for this system.  Note that in the isentropic case with $P(\rho) = \rho^{\gamma}$ ($\gamma > 1$), it is enough to assume that the initial values are bounded since solutions can be constructed conserving this property (see Chen \cite{Chen3}, or Lions Perthame Tadmor \cite{LPT}, for instance). 
%In the special case $\gamma=3$, those solutions verify the strong trace property \cite{Vasseur_gamma3},  %the stability theory is complete and the system is well-posed. 
\vskip0.3cm
We now show an application of our method in the general setting of strictly hyperbolic conservation laws with  genuinely nonlinear characteristic fields.
We consider an $n \times n$ system of conservation laws
\begin{align}\label{cl1}
\dt U+\dx A(U)=0,
\end{align}
which has a strictly convex entropy $\eta$. Assume that $A$ and $\eta$ are of class $C^2$ on an open state domain $\mV \subset \R^n$.
We have the following result.

\begin{theo}\label{theoHyper}
%Assume that the system (\ref{cl1}) is strictly hyperbolic and that the 1st characteristic family (resp. nth) is genuinely %nonlinear or linearly degenerate. 
%Assume that the 1-characteristic family (resp. n-characteristic family) of (\ref{cl1}) is genuinely nonlinear or linearly %degenerate and that the characteristic speed $\lambda_1(V)$ (resp. $\lambda_n(V)$) is a simple eigenvalue of %$\nabla A(V)$ for all $V \in \mV$. 
Assume that the smallest (resp. largest) eigenvalue of $\nabla A(V)$ is simple for all $V \in \mV$, and that the corresponding 1-characteristic family (resp. n-characteristic family) of (\ref{cl1}) is  genuinely nonlinear.  Then, for any $V_0 \in \mV$, there exists $K>0$, $a>0$, and  $C>0$ such that, for any entropy-admissible 1-shock  (resp. n-shock) with speed $\sigma$ and endstates $(U_L, U_R)$ verifying $U_L \in B_K(V_0)$ and $U_R \in B_K(V_0)$, the following is true.
For any weak entropic solution $U$ bounded in $B_K(V_0)$ on $(0,T)$ (with possibly $T=+\infty$), 
 there exists a Lipschtiz path $x(t)$ such that for any $t<T$, the pseudo norm
$$
\int_{-\infty}^0\eta(U(t,x+x(t))| U_L)\,dx+a\int_0^\infty \eta(U(t,x+x(t)|U_R)\,dx
$$
is  non increasing in time. Moreover, for every $t<T$:
\begin{eqnarray*}
&&|x'(t)|\leq C,\\
&&|x(t)-\sigma t|\leq C\sqrt{t}\|U_0-S\|_{L^2(\R)},
\end{eqnarray*}
where $U_0=U(t=0)$, and $S(x)=U_L$ for $x<0$ and $S(x)=U_R$ for $x>0$.
\vskip0.3cm
Especially, we have for every $t>0$
$$
\|U(t,\cdot+x(t))- S\|_{L^2(\R)}\leq C \|U_0- S\|_{L^2(\R)}.$$
\vskip0.3cm
Finally, we can choose $a<1$ of a 1-shock, and $a>1$ for a n-shock.
\end{theo}

In particular, this provides $L^2$ stability, up to a drift, for suitably weak shocks in a class of perturbations without $BV$ conditions. Note that the assumption of genuinely nonlinearity  applies only to the wave family associated to the extremal eigenvalue. No such assumptions are needed on the other wave families.
\vskip0.3cm
At least in the scalar case, the estimate on the drift $x(t)$ is rather precise.
We will show the following proposition.
\begin{prop}\label{prop_xt}
Consider the Burgers equation
$$
\dt u+\dx u^2=0,
$$
and the steady shock solution
$$
S(x)=1 \ \ \mathrm{for} \ \  x<0, \mathrm{\ \ and \ \ } S(x)=-1 \ \ \mathrm{for} \ \  x>0.
$$
For any $p>1/2$, there exists a constant $C_p$ such that for any  $\eps$ small enough, there exists a initial value $u^0$ such that the associated unique entropic solution $u$ has a  
drift $t\to x(t)$ verifying that
$$
\int_\R|u(t,x)-S(x-x(t))|^2\,dx \mathrm {\ \ is \ \  non \ increasing.}
$$
Moreover, for any such drift and $t>1$
$$
x(t)\geq C_p \sqrt{\eps} t^p, \qquad \mathrm{with} \ \ \eps=\int_\R|u^0(x)-S(x)|^2\,dx.
$$
\end{prop}

The theorems above highlight only a few applications of our theory. 
In the next section, we develop our methods in a more general framework.
The assumptions on the Hugoniot curves are quite natural and we require no smallness condition on the discontinuities at play. We can even relax the strict hyperbolicity condition and consider cases in which the middle eigenvalues degenerate and possibly cross each other.

\section{Presentation of the results}\label{section_presentation}

\subsection{General framework}

We want to study a system of $m$ equations of the form
\begin{equation}\label{system}
\dt U + \dx A(U)=0,
\end{equation}
where the flux function $A$ is defined on an open,  convex set $\mV \subset \R^m$. 
$$
A: \mV\subset \R^m\longrightarrow \R^m.
$$
We assume that $A \in C^2(\mV)$. We assume that the system is hyperbolic on $\mV$. That is, for any $U\in \mV$, the $m\times m$ matrix $\nabla A(U)$ is diagonalizable. 
We denote by $\lambda^-(U)$ and $\lambda^+(U)$ the smallest and largest eigenvalues, respectively, of $\nabla A(U)$. Hereafter, we assume that  $\lambda^\pm(U)$ are simple eigenvalues for all $U\in \mV$ (But  we do not make such hypothesis for the other eigenvalues).

Additionally, we assume the existence of a strictly convex entropy
$$
\eta: \mV\subset \R^m\longrightarrow \R,
$$
of class $C^2$, and an associated entropy flux
$$
G: \mV\subset \R^m\longrightarrow \R,
$$
of class $C^2$, such that the following compatibility relation holds on $\mV$.
\begin{equation}\label{entropy flux}
\partial_j G=\sum_{i=1}^m \partial_i \eta \thinspace \partial_j A_i \qquad \mathrm{for \  any \ } 1\leq j\leq m.
\end{equation}

If we want to apply our theory to the systems of gas dynamics, we have to define these functions on a suitable subset of the boundary of $\mV$, namely the points corresponding to vacuum states. 
%In fact, the vacuum state for the Euler systems will correspond to a single point in $\mU \setminus \mV$.
For this reason, we introduce  as in  \cite{Vasseur_Book}
\begin{equation*}\label{mU}
\mU=\{V\in\R^m \ \vert \ \ \exists V_k\in \mV, \ \lim_{k\to\infty} V_k=V, \ \limsup_{k\to\infty} \ \eta(V_k)<\infty\},
\end{equation*}
and extend the entropy functional $\eta$ on $\mU$ by
$$
\eta(\overline{U})=\liminf_{\mV\ni U\to \overline{U}}\eta(U).
$$
Note that if $\eta$ is unbounded on $\mV$, $\mU$ can be strictly smaller than $\overline{\mV}$. This happens for  the Euler system. In this case, the non vacuum states are  $(\rho,\rho u, \rho E)\in \mV=(0,\infty)\times \R\times (0,\infty)$. And $\mU=\mV\cup \{(0,0,0)\}$, while $\overline{\mV}=[0,\infty)\times\R\times[0,\infty)$ includes non physical states. In the general case, $\mU$ is still convex, and $\eta$ is convex on $\mU$ (see \cite{Vasseur_Book}).  We denote by $\mU^0$ the subset of $\mU$ where at least one of the functions $\eta$, $A$, $G$ is not $C^1$ (typically the vacuum states). 
Still, $A$, $\eta$, and $G$ may not be even continuous on $\mU$ (as for the Euler system because of large velocities).  We will restrict our study to weak solutions whose values are in a subset $\mU_K$ verifying
\begin{equation}\label{1*}
\begin{array}{l}
 \mU_K \mathrm {\  is \ a \ convex \ bounded \ subset \  of  \ }  \mU,\\
\mathrm{The \  functions \ } A, \ \eta, \ \mathrm{and} \ G \mathrm{\  are\  continuous \ on \  } \mU_K,\\
\mathrm{ \  (possibly \ with \ no \   additional \  regularity \ up \  to \  the \  boundary).}
\end{array}
\end{equation}
We consider cases where $\mU^0\subset \mU_K$. In this case, $A$, $\eta$, and $G$ may not be $C^1$ on $\mU_K$.  The system may even fail to be hyperbolic on $\mU_K$. Indeed, the eigenvalues may be undefined on $\mU^0$.
\vskip0.3cm
We will be careful to show that the pseudo norm does not depend on $\mU_K$. Only the drift does. 
Note that for the Euler system, we can consider 
$
\mU_K=\{U=(\rho,\rho u,\rho E) \ \setminus \  \sup(\rho,|u|,|E|)\leq K, \ \ \rho\geq0\},
$
which includes the vacuum $\mU^0=\{(0,0,0)\}$. 
\vskip0.5cm
\vskip0.5cm
Next, we define, for any $V\in \mV$, $U\in\mU$, the relative entropy function
%For any $V\in \mV$, $U\in\mU$, we define the relative entropy function as
$$
\eta(U \midd V)=\eta(U)-\eta(V)- \nabla \eta(V)\cdot(U-V).
$$
Since $\eta$ is convex on $\mU$ and strictly convex in $\mV$, we have (see \cite{Vasseur_Book})
$$
\eta(U \midd V)\geq 0,\qquad U\in\mU, \ V\in \mV,
$$
and
$$
\eta(U \midd V)=0 \quad \mathrm{if \ and \ only \ if} \quad U=V.
$$
The following lemma shows that the relative entropy is comparable to the square of the $L^2$ norm on $\mU_K$ (it is usually not true on $\mU$).
\begin{lemm}\label{L2}
For any compact set $\Omega \subset \mV$, there exist $C_1, C_2 > 0$ (depending both on $\Omega$ and $\mU_K$) such that 
$$
C_1|U-V|^2\leq\eta(U \midd V)\leq C_2|U-V|^2,
$$
for any $U\in \mU_K$ and  $V \in \Omega$.
\end{lemm}
A proof of this lemma can be found in \cite{Vasseur_Book} and in \cite{LV}. Obviously, the lemma holds also for $(U,V)\in \Omega^2$, $\Omega$ compact set of $\mV$.
\vskip0.3cm

For a pair of states $U_L \ne U_R$ in $\mV$, we say that $(U_L,U_R)$ is an entropic Rankine-Hugoniot discontinuity if there exists $\sigma\in \R$ such that 
\begin{equation}\label{RH}
\begin{array}{l}
\ds{A(U_R)-A(U_L)=\sigma(U_R-U_L),}\\[0.3cm]
\ds{G(U_R)-G(U_L)\leq\sigma(\eta(U_R)-\eta(U_L))}.
\end{array}
\end{equation}
%This implies that the discontinuous function $U$ defined by 
Equivalently, this means that the discontinuous function $U$ defined by
\begin{align*}
U(t,x) = 
\begin{cases}
U_L, &\text{if $x<\sigma t$,}\\[0.1 cm]
U_R, &\text{if $x>\sigma t$,}\\
\end{cases}
\end{align*}
%\begin{eqnarray*}
%U(t,x)&=&U_L\qquad \mathrm{if} \  x<\sigma t,\\
%&=&U_R\qquad \mathrm{if} \  x>\sigma t,
%\end{eqnarray*}
is a weak solution to (\ref{system}) verifying also, in the sense of distributions, the entropy inequality
\begin{equation}\label{entropie}
\dt \eta(U)+\dx G(U)\leq 0.
\end{equation}
%Finally we denote by $\lambda^-(U)$ (respectively $\lambda^+(U)$) the smallest eigenvalue of $\nabla A(U)$ %(respectively the largest eigenvalue of $\nabla A(U)$). 

\subsection{Hypotheses on the system}

We take the same set of hypotheses as in \cite{LV}, except that we consider only the case of shocks (not contact discontinuities), and strict inequality in the Liu conditions (H1)(a) and (H1)(b) (respectively (H1')(a) and (H1')(b)).
\vskip0.3cm

First, we assume that for any $(U_-, U_+)$ entropic Rankine-Hugoniot discontinuity with $U_-\neq U_+$ we have both $U_-\notin \mU^0$ and   $U_+\notin \mU^0$. (Typically, there is no shock connecting the vacuum.) 
\vskip0.3cm
We will consider two sets of assumptions. One set will imply the result on the 1-shock, the second set (dual from the first one)
will imply the result on the n-shock. A system satisfying both set of hypotheses, verifies both results. 
\vskip0.4cm
\noindent{\bf First set of hypotheses}
\vskip0.1cm 

The first set of hypotheses, related to some  $U_L\in \mV$, is the following ((H1) to (H3)). 
\begin{itemize}
\item[(H1)] (Family of 1-shocks verifying the Liu condition)\\
There exists a neighborhood $B\subset\mV$ of $U_L$ such that for any $U\in B$, there is a one parameter family of states $S_U(s)\in \mU$ defined on an interval $[0,s_U)$ (with possibly $s_U=\infty$),  such that  $S_{U}(0)=U$, and 
$$
A(S_U(s))-A(U)=\sigma_U(s)(S_U(s)-U), \qquad s\in [0,s_U),
$$
(which means that $(U,S_U(s))$ is a Rankine-Hugoniot discontinuity with velocity $\sigma_U(s)$). We assume that $U\to s_U$ is Lipschitz on $B$ and 
both $(s,U)\to S_U(s)$ and $(s,U)\to\sigma_U(s)$ are $C^1$ on $\{(s,U) | U\in B, 0\leq s< s_U\}$. We assume also that the following properties hold for $U\in B$.
\begin{itemize}
\item[(a)] $\sigma_U'(s)< 0$ for $0\leq s < s_U$ (the speed of the shock decreases with $s$), and $\sigma_U(0) = \lambda^-(U)$.
\item[(b)]   $\ds{\frac{d}{ds}\eta(U | S_U(s))}>0$ (the shock "strengthens" with $s$) for all $s$.
\end{itemize}
\item[(H2)] 
%For any entropic Rankine-Hugoniot discontinuity $(U,V)$ with velocity $\sigma$ such that  $V\in B$, we  have %$\sigma\geq\lambda^-(V)$. 
If $(U,V)$ is an entropic Rankine-Hugoniot discontinuity, $U\neq V$, with velocity $\sigma$ such that $V \in B$, then $\sigma\geq\lambda^-(V)$.
\item[(H3)]
%For any entropic Rankine-Hugoniot discontinuity $(U,V)$ with velocity $\sigma$ which is not a $1$-shock (or 1-contact discontinuity) and for which $U\in B$, we have $\sigma\geq\lambda^-(U)$.
If $(U,V)$ is an entropic Rankine-Hugoniot discontinuity with velocity $\sigma$ such that $U \in B$ and $\sigma < \lambda^-(U)$, then $(U,V)$ is a 1-shock. In particular, $V = S_U(s)$ for some $0\leq s< s_U$.
\end{itemize}
\vskip0.4cm

\noindent{\bf Second set of hypotheses}
\vskip0.1cm

The second set of hypotheses, related to some  $U_R\in \mV$, is the following (($\text{H}^{\prime}$1)  to ($\text{H}^{\prime}$3)). 
\begin{itemize}
\item[($\text{H}^{\prime}$1)] (Family of  $n$-shocks verifying the Liu condition)\\
There exists a neighborhood $B\subset\mV$ of $U_R$  such that  for every $U\in B$ there is a one parameter family of states $S_U(s)\in \mU$ defined on an interval $[0,s_U)$ (with possibly $s_U=\infty$),  such that $S_{U}(0)=U$, and 
$$
A(S_U(s))-A(U)=\sigma_U(s)(S_U(s)-U), \qquad s\in [0,s_U),
$$
(which means that $(S_U(s),U)$ is a Rankine-Hugoniot discontinuity with velocity $\sigma_U(s)$). We assume that $U\to s_U$ is Lipschitz and 
both $(s,U)\to S_U(s)$ and $(s,U)\to\sigma_U(s)$ are  $C^1$ on $\{(s,U)| U\in B, s\in [0,s_U)\}$. We assume also that the following properties hold for $U\in B$.
\begin{itemize}
\item[(a)] $\sigma_U'(s)> 0$ for $0\leq s< s_U$ (the speed of the shock increases with $s$), and $\sigma_U(0)=\lambda^+(U)$.
\item[(b)] $\ds{\frac{d}{ds}\eta(U | S_U(s))}>0$ (the shock ''strengthens''  with $s$) for all $s$.
\end{itemize}
\item[($\text{H}^{\prime}$2)] 
%For any entropic Rankine-Hugoniot discontinuity $(V,U)$ with velocity $\sigma$ such that  $V\in B$, we have %$\sigma\leq\lambda^+(V)$.
If $(V,U)$ is an entropic Rankine-Hugoniot discontinuity, $V\neq U$, with velocity $\sigma$ such that $V \in B$, then $\sigma\leq\lambda^+(V)$.
\item[($\text{H}^{\prime}$3)] 
%For any entropic Rankine-Hugoniot discontinuity $(V,U)$ with velocity $\sigma$ which is not an $n$-shock (or $n$-%contact discontinuity) and for which $U\in B$, we have $\sigma\leq\lambda^+(U)$.
If $(V,U)$ is an entropic Rankine-Hugoniot discontinuity with velocity $\sigma$ such that $U \in B$ and $\sigma > \lambda^+(U)$, then $(V,U)$ is an n-shock. In particular, $V = S_U(s)$ for some $0\leq s< s_U$.
\end{itemize}
\vskip0.5cm
\noindent{\bf Remarks}
\vskip0.1cm

\begin{itemize}
\item Note that a given system (\ref{system}) verifies Properties (H1) to (H3) relative to $U \in \mV$ if and only if  the system 
\begin{equation}\label{reverse system}
 \dt U-\dx A(U)=0,
\end{equation}
verifies Properties ($\text{H}^{\prime}$1) to ($\text{H}^{\prime}$3) relative to the same $U \in \mV$. The properties are, in this way, dual. 
\item Property (H1) assumes the existence of a family a 1-shocks $(U, S_U(s))$ verifying the Liu entropy condition for all $s$ (Property (a)). The only additional requirement is (b), which is a condition on the growth of the shock along $S_U(s)$, where the growth is measured
through the pseudo-metric induced from the entropy. 
%Note that this assumption is quite natural. Note that a variant of the requirement (b) can be found in 
%Dafermos \cite{Dafermos4} in a completely different context. 
This condition arises naturally in the study of admissibility criteria for systems of conservation laws. In particular, it ensures that Liu admissible shocks are entropic even for moderate to strong shocks.
%(see, for instance, \cite{Dafermos4,Lax,Ruggeri}).
Indeed, this fact follows from the important formula
\begin{align*}
G(S_{U_L}(s))-G(U_L) = \sigma_{U_L}(s) \thinspace (\eta(S_{U_L}(s)) - \eta(U_L)) + \int_0^s\sigma'_{U_L}(\tau)\eta(U_L \midd S_{U_L}(\tau))\,d\tau.
\end{align*}
(See also \cite{Dafermos4,Lax,Ruggeri,LV}.)

\item Hypothesis (H2) is fulfilled under the very general assumption that all the entropic Rankine-Hugoniot discontinuities verify the Lax entropy conditions, that is
$$
\lambda_i(U_-) \geq \sigma \geq \lambda_i(U_+),
$$
for any $i$-shocks $(U_-,U_+)$ with velocity $\sigma$ and any $1\leq i\leq n$. Indeed, we need only the second inequality, and the fact that 
$\lambda_i(U_+)\geq \lambda^-(U_+)$.

\item Hypothesis (H3) is a requirement that the family of 1-discontinuities is well-separated from the other Rankine-Hugoniot discontinuities and do not interfere with them. In the case of strictly hyperbolic systems,
it is, for instance, a consequence of the extended Lax admissibility condition%popular stability condition 
$$
\lambda_{i+1}(U_+) \geq \sigma \geq \lambda_{i-1}(U_-),
$$
for all i-shocks $(U_-,U_+)$, $i>1$, with velocity $\sigma$.
Indeed, we use only the second inequality and the fact that $\lambda_{i-1}(U_-)\geq\lambda^-(U_-)$. Note that we need to separate only the $1$-shocks
issued from $B$, that is close to $U_L$. 
\item The existence of an entropy $\eta$ implies that the system (\ref{system}) is hyperbolic. Since $A\in C^2(\mV)$, the eigenvalues of $\nabla A(U)$ vary continuously on $\mV$. In particular, since $\lambda^\pm(U)$ are simple for $U\in\mV$, the implicit function theorem ensures that the maps $U\to\lambda^\pm(U)$ are in $C^1(\mV)$. 
Note, however, that those maps may be discontinuous on $\mU$.
\item In \cite{BFZ}, Barker, Freist\"uhler, and Zumbrun   showed that the stability (and so the contraction as well) fails to hold for the full Euler system if Hypothesis (H1 (b)) is replaced by 
$$
\ds{\frac{d}{ds}\eta(S_U(s))}>0, \qquad \mathrm{for \ all \ }s.
$$
It shows that the strength of the shock is better measured by the relative entropy rather than the entropy itself. 
\end{itemize}

\subsection{Statement of the result}

Our main result is the following.
\begin{theo}\label{main}
Consider a system of conservation laws (\ref{system}), such that $A$ is $C^2$ on  an open convex subset  $\mV$ of $\R^m$.
We assume that there exists  a $C^2$ strictly convex entropy $\eta$ on $\mV$ verifying (\ref{entropy flux}). Let $U_L, U_R\in \mV$ such that either the  system (\ref{system}) verifies the Properties (H1)--(H3) and there exists $s>0$ such that $U_R=S_{U_L}(s)$ and $\sigma=\sigma_{U_L}(s)$ (so  $(U_L,U_R)$ is a 1-shock with velocity $\sigma$), or the  system (\ref{system}) verifies the Properties ($\text{H}^{\prime}$1)--($\text{H}^{\prime}$3) and there exists $s>0$ such that $U_L=S_{U_R}(s)$ and $\sigma=\sigma_{U_R}(s)$ (so   $(U_L,U_R)$ is a $n$-shock with velocity $\sigma$). Then, there exists $a>0$ with the following property. For any bounded convex subset $\mU_K$ of $\mU$ on which   $\eta$, $A$ and $G$ are continuous, there exists a constant $C_K>0$ such that the following holds true.  For  any weak entropic solution $U$ of (\ref{system}) with values in $\mU_K$ on $(0,T)$ (with possibly $T=\infty$) verifying the strong trace property of Definition \ref{defi_trace}, there exists a Lipschitzian map $x(t)$ such that for any $0<t<T$ the pseudo norm
$$
\int_{-\infty}^0\eta(U(t,x+x(t))| U_L)\,dx+a\int_0^\infty \eta(U(t,x+x(t)|U_R)\,dx
$$
is  non increasing in time. Moreover, for every $0<t<T$:
\begin{eqnarray*}
&&|x'(t)|\leq C_K,\\
&&|x(t)-\sigma t|\leq C_K\sqrt{t}\|U_0-S\|_{L^2(\R)},
\end{eqnarray*}
where $U_0=U(t=0)$, and $S(x)=U_L$ for $x<0$ and $S(x)=U_R$ for $x>0$.
\vskip0.3cm
Especially, we have for every $t>0$
$$
\|U(t,\cdot+x(t))- S\|_{L^2(\R)}\leq C_K \|U_0- S\|_{L^2(\R)}.$$
\vskip0.3cm
Finally, we can choose $a<1$ of a 1-shock, and $a>1$ for a n-shock.
\end{theo}
Note that the pseudo norm  (and $a$ which defines it) does not depend on $\mU_K$. Therefore, it does not depend on any quantitative property of $U$ (especially, not on its $L^\infty$ norm). 
\vskip0.3cm
The correction of the position of the approximated shock $x(t)$ is fundamental, since the 
result is trivially wrong without it, even for Burgers' equation in the scalar case (see \cite{Leger}). Part of the difficulty of the proof is to find this correct position.
\vskip0.3cm
Note that it is enough to show the result for a 1-shock. The result for n-shocks is obtained applying the result for 1-shocks on  $\tilde{U}(t,x)=U(t,-x)$, which is an entropic solution to (\ref{reverse system}).
In particular, if we consider a $2\times2$ system which verifies both (H1)--(H3) for any $U_L\in\mV$ and ($\text{H}^{\prime}$1)--($\text{H}^{\prime}$3) for any $U_R\in\mV$, then
all shocks  are unique and stable in $L^2$. 
\vskip0.3cm
Note that the assumptions on the system are quite minimal. There is an assumption only on the wave coming from $U_L$ (or going to $U_R$  for n-shocks). There are absolutely no assumptions on the other waves (which may not even exist or may be neither genuinely nonlinear nor linearly degenerate). The extremal property that the shock curve under consideration corresponds to the smallest eigenvalue of $\nabla A(U)$ (resp. the largest 
eigenvalue of $\nabla A(U)$) is only prescribed on a small neighborhood of $U_L$ (resp. of $U_R$). 
%Not only can the system fail to be strictly hyperbolic, but also the eigenvalues may cross each other in $\mV$. 
Finally the theory allows us (via the extended set $\mU$) to consider weak solutions which
may take values $U$ corresponding to points of non-differentiability of $A$ and $\eta$. This includes, for example, the vacuum states in fluid mechanics. It has been verified in \cite{LV} that the isentropic Euler system, the full Euler system, and the general case stated in the introduction verify the Hypotheses (H1)--(H3) and ($\text{H}^{\prime}$1)--($\text{H}^{\prime}$3). Therefore theorems (\ref{theoEuler}), and (\ref{theoHyper}) are consequences of Theorem \ref{main}.

\subsection{Main ideas of the proof}

We will restrict our proof to the case of a 1-shock. The result on n-shock is a direct consequence of it as explained in the previous section. 
 The following estimate underlies most of our analysis.
\begin{lemm}\label{defi_F}
If $V\in\mV$ and $U$ is any weak entropic solution of (\ref{system}), then $\eta(U \midd V)$ is a solution in the sense of distributions to
$$
\dt \eta(U \midd V)+\dx F(U, V)\leq 0,
$$
where 
$$
F(U,V)=G(U)-G(V)-\nabla \eta(V) \cdot (A(U)-A(V)).
$$
\end{lemm}
The proof of this lemma is direct from the definition of the relative entropy (Note that $V$ is constant with respect to $t$ and $x$, and so $U\to \eta(U|V)$ is still a convex entropy for the system). For a given shift $t\to x(t)$, and $a>0$, Let us denote
\begin{equation}\label{def_E}
\E(t)=\int_{-\infty}^{x(t)}\eta(U(t,x)|U_L)\,dx+a\int_{x(t)}^{+\infty}\eta(U(t,x)|U_R)\,dx.
\end{equation}
From  Lemma \ref{defi_F}, and the strong trace property of Definition \ref{defi_trace}, we will show that
\begin{equation}\label{evolution_E}
\begin{array}{l}
\ds{\frac{d}{dt}\E(t) \leq x'(t)\left[\eta(U(t,x(t)-) \midd U_L)-a\eta(U(t,x(t)+) \midd U_R)\right]}\\[0.3cm]
\qquad \qquad \ds{-F(U(t,x(t)-),U_L)+aF(U(t,x(t)+),U_R)} ,
\end{array}
\end{equation}
for almost every $t$. 
The idea, is to construct a shift on the fly, via an ODE, in order to make this contribution non positive. 
\vskip0.3cm 
Let us focus, first, on the situation when $U$ is Lipschitz. In particular we have $U(t,x(t)-)=U(t,x(t)+)=U(t,x)$. When 
$$
\eta(U(t,x(t)) \midd U_L)-a\eta(U(t,x(t)) \midd U_R)=0,
$$
the shift has no effect on the evolution of $\E(t)$. When $a=1$, this corresponds to values of $U$ lying in a whole hyperplane in $\mV$.  For general system (including Euler systems), the contribution 
$$
-F(U,U_L)+F(U,U_R)
$$
is not globally non positive on this hyperplane (see \cite{SV}). However, for $a$ small enough, the set
$$
\Oa=\{ U \ \setminus (\eta(U \midd U_L)-a\eta(U \midd U_R))\leq 0\ \}
$$
is contained in a small ball centered at $U_L$, let say $B(U_L,C_0/2)$ (at the limit $a\to 0$, this converges to the point $U_L$). 
A key observation (Lemma \ref{lemm_def_C0}) is that, whenever the shock $(U_L, U_R)$ with velocity  $\sigma$ is a 1-shock, there exists $v\in (\sigma, \lambda_-(U_L))$ such that the dissipation terms verify
 \begin{equation}\label{eq_good}
 \begin{array}{l}
 -F(U,U_L)+v\eta(U \midd U_L)<0,\\[0.2cm]
F(U,U_R)-v\eta(U \midd U_R)<0,
\end{array}
\end{equation}
on $B(U_L,C_0)$, for $C_0, a$ small enough. 

Then, it is natural to construct the shift in the following way:
% We define first a smooth solution $V$ on $\mV$ verifying
$$
 V(U)=v -\frac{[-F(U,U_L)+v\eta(U\midd U_L)]_++a[F(U,U_R)-v\eta(U\midd U_R)]_+}{\eta(U \midd U_L)-a\eta(U \midd U_R)}, \qquad\mathrm{for} \ \  U\in \mV,
$$
%\begin{eqnarray*}
%&&V(U)= v, \mathrm{for \ \ U\in \Oa},\\
%&& V(U)=v -\frac{[-F(U,U_L)+v\eta(U\midd U_L)]_++a[F(U,U_R)-v\eta(U\midd U_R)]_+}{\eta(U \midd U_L)-a\eta(U \midd U_R)}, \ \ \ \ \mathrm{U\in (B(U_L,C_0))^c},\\
%&&V(U)\leq v\qquad \mathrm{for \ \ all \ \ }U\in \mV,
%\end{eqnarray*}
%where $(B(U_L,C_0))^c$ is the complement of $B(U_L,C_0)$ in $\mV$, and
where  $[\cdot]_+=\sup(0,\cdot)$.
Then we define $x(t)$ through the ODE:
\begin{equation}\label{ideal}
\dot{x}(t)=V(U(x(t))),\qquad x(0)=0. 
\end{equation}
The function $U\to V(U)$ is well defined on $\mV$ since the numerator vanishes for 
$
U\in B(U_L,C_0) 
$
which contains the set $\{U \setminus \eta(U \midd U_L)-a\eta(U \midd U_R)\}$. Especially, 
$V(U)=v$ for $U\in B(U_L,C_0)$, and so, also for $U\in \mathcal{O}_a$.
Note that whenever $U$ is smooth and is valued in $\mV$, we can solve this ODE in a unique way, and the construction ensures that 
$$
\frac{d}{dt}\E(t) \leq 0.
$$
%This is obvious whenever  $U(t,x)\in \Oa\cup (B(U_L,C_0))^c$. This is still true when $U(t,x)\in B(U_L,C_0)\setminus \Oa$ since 
%for those values, $V(U(t,x))\leq v$, (\ref{eq_good}) is still valid, and $\eta(U(t,x)\midd U_L)-a\eta(U(t,x)\midd U_R)\geq0$.
\vskip0.3cm
Of course, when the solution is discontinuous, (or have values in  $\mU^0$ (the ``vacuum")),  (\ref{ideal}) cannot be solved in the classical sense. Hence we can define $x(t)$ only in the Filippov way. We will have to check carefully that we can do it using only the strong trace property of Definition \ref{defi_trace}. Even so, we cannot ensure that (\ref{ideal}) holds almost everywhere. However, we will use the fact that for almost every time $t$, especially when $U(t,x(t)+) \neq U(t,x(t)-)$, the following Rankine--Hugoniot relation holds:
$$
A(U(t,x(t)+)-A(U(t,x(t)-)=x'(t)(U(t,x(t)+)-U(t,x(t)-)).
$$ 
So, we have to investigate the value of (\ref{evolution_E}) whenever $(U(t,x(t)-), U(t,x(t)+))$ is an entropic discontinuity with velocity $x'(t)$. Note that this case where the drift $x(t)$ is stuck in a shock is quite generic (see the special example in section \ref{scalar}).  We show that whenever 
$U(t,x(t)-)$ and $U(t,x(t)+)$ are both outside $\Oa$, the situation is similar to the continuous case. If one of them is in $\Oa$, using the fact that $(U_L,U_R,\sigma)$ is a 1-shock, and $x'(t)\leq v$, we get that $U(t,x(t)-)$ is in $\Oa$, and $(U(t,x(t)-), U(t,x(t)+))$ is itself a 1-shock with velocity $x'(t)$. If $U(t,x(t)-)=U_L$, then the result comes from a key structural lemma first proved by DiPerna in \cite{DiPerna} (see also \cite{LV}).
Using the dissipation of the shock $(U_L,U_R)$  with velocity  $\sigma$, we show that for $a$ small enough, (\ref{evolution_E}) is still non positive for $U(t,x(t)-)\in\Oa$ whenever $U(t,x(t)+)$ is on the 1-shock curve (even if this curve is unbounded in $\mV$).

\vskip0.3cm
The rest of the paper is organized as follows. In the next section, we prove the main structural lemmas. They do not depend on the solutions $(t,x)\to U(t,x)$, but only the properties of the system. We construct $a$ in this section. Notice that the results of this section do not depend on $\mU_K$ (and so, do not depend on any quantitative bound on the solutions themselves). 
In the following section we construct the path $t\to x(t)$, which depends on $\mU_K$. The next one is dedicated to the proof of the main theorem. In the last section, we prove Proposition \ref{prop_xt}. 

\section{Construction of the pseudo-norm}\label{structure}
The pseudo-norm, based on the relative entropy, is not anymore  homogeneous in $x$. It depends only on the number $a>0$:
\begin{eqnarray*}
d(U,S)(x)&=&\eta(U(x)|U_L) \qquad \mathrm{for} \ x<\sigma t,\\
&=&a \eta(U(x)|U_R) \qquad \mathrm{for} \ x>\sigma t,
\end{eqnarray*}
where $S$ is the fixed shock $(U_L,U_R)$ with velocity $\sigma$. This section is dedicated to the construction of this number $a$. Results in this section do not depend on any particular weak entropic solution $U$  (and so, do not depend on the set $\mU_K$). The results depend only on values of quantities in the state space $\mU$.
%The first lemma can be traced back to the work of Lax \cite{Lax}

The first lemma of this section gives an explicit formula for the entropy lost at a Rankine--Hugoniot discontinuity $(U_-,U_+)$, where $U_+=S_{U_-}(s)$ for some $s>0$. The estimate can be traced back to the work of Lax \cite{Lax}.
%1-discontinuity (that is, a discontinuity $(U_-,U_+)$ such that $U_+=S_{U_+}(s)$ for a certain $s>0$). It can be traced back to the work of Lax \cite{Lax}.

\begin{lemm}\label{decreasing}
Assume $(U_-,U_+)\in \mV^2$ is an entropic Rankine-Hugoniot discontinuity with velocity $\sigma$; that is, $(U_-,U_+)$ verifies (\ref{RH}). Then, for any $V\in \mU$
$$
F(U_+,V)-\sigma \eta(U_+ \midd V)\leq F(U_-,V)-\sigma\eta(U_- \midd V),
$$
where $F$ is defined as in Lemma \ref{defi_F}. Furthermore, if $U_-\in B$, as in Hypothesis (H1), and there exists $s>0$ such that $U_+=S_{U_-}(s)$ and $\sigma=\sigma_{U_-}(s)$ (that is, $(U_-,U_+)$ is a $1$-discontinuity), then
$$
F(U_+,V)-\sigma \eta(U_+ \midd V)= F(U_-,V)-\sigma\eta(U_- \midd V)+\int_0^s\sigma'_{U_-}(\tau)\eta(U_- \midd S_{U_-}(\tau))\,d\tau.
$$
\end{lemm}
\begin{proof}
Since $(U_-,U_+)\in \mV^2$ is an entropic Rankine-Hugoniot discontinuity with velocity $\sigma$ we have
$$
-\nabla \eta(V) \cdot (A(U_+)-A(U_-))=-\sigma \nabla \eta(V) \cdot (U_+-U_-),
$$
and 
$$
G(U_+)-G(U_-)\leq \sigma (\eta(U_+)-\eta(U_-)).
$$
Summing those two estimates gives the first result.

Assume now that it is a $1$-discontinuity.
Then, define
\begin{align*}
\mathcal{F}_1(s)&=F(S_{U_-}(s),V)-F(U_-,V),\\
\mathcal{F}_2(s)&=\sigma_{U_-}(s)(\eta(S_{U_-}(s) \midd V)-\eta(U_- \midd V))+\int_0^s\sigma'_{U_-}(\tau)\eta(U_- \midd S_{U_-}(\tau))\,d\tau.
\end{align*}
We want to show that $ \mathcal{F}_1(s)= \mathcal{F}_2(s)$ for all $s$.
Since $S_{U_-}(0)=U_-$, the equality is true  for $s=0$.
Next we have
\begin{align*}
\mathcal{F}'_1(s)&=\frac{d}{ds}G(S_{U_-}(s))-\nabla \eta(V) \cdot \frac{d}{ds}A(S_{U_-}(s))\\[0.1 cm]
&= [ \nabla \eta(S_{U_-}(s))-\nabla \eta(V) ] \cdot \frac{d}{ds}[A(S_{U_-}(s))-A(U_L)],
\end{align*}
and 
\begin{align*}
\mathcal{F}'_2(s)&=\sigma_{U_-}'(s)[\nabla \eta(V) \cdot ((S_{U_-}(s)-V)-(U_--V))-\nabla \eta(S_{U_-}(s)) \cdot (S_{U_-}-U_-)]\\[0.1 cm]
&\qquad \qquad \qquad \qquad \qquad +\sigma_{U_-}(s)[\nabla \eta(S_{U_-}(s))-\nabla \eta(V)] \cdot S_{U_-}'(s)\\[0.2 cm]
&=[\nabla \eta(S_{U_-}(s))-\nabla \eta(V)] \cdot \frac{d}{ds}[\sigma_{U_-}(s) (S_{U_-}(s)-U_L)].
\end{align*}
Using the fact that $(U_-,S_{U_-}(s))$ with velocity $\sigma_{U_-}(s)$ verifies the Rankine-Hugoniot conditions, we get
$$
\mathcal{F}'_1(s)=\mathcal{F}'_2(s)\qquad \mathrm{for\ } s>0.
$$
\end{proof}

The next lemma is a variation on a crucial
lemma of DiPerna \cite{DiPerna}. It is an extension of a lemma from  \cite{LV}.
\begin{lemm}\label{cornerstone}
For any $U\in B$ and any $s>0$, $s_0>0$, we have
$$
F(S_U(s),S_U(s_0))-\sigma_U(s)\eta(S_U(s) \midd S_U(s_0))=\int_{s_0}^{s}\sigma_U'(\tau)(\eta(U|S_U(\tau))-\eta(U \midd S_{U}(s_0)))\,d\tau\leq0.
$$
Especially, there exists $\delta>0$ and $\kappa>0$ such that we have the following.
\begin{eqnarray*}
&& F(S_U(s),S_U(s_0))-\sigma_U(s)\eta(S_U(s) \midd S_U(s_0))\leq -\kappa |\sigma_U(s)-\sigma_U(s_0)|^2, \qquad \mathrm{for \ \ } |s-s_0|\leq \delta,\\
&& F(S_U(s),S_U(s_0))-\sigma_U(s)\eta(S_U(s) \midd S_U(s_0))\leq -\kappa |\sigma_U(s)-\sigma(s_0)|, \qquad \mathrm{for \ \ } |s-s_0|\geq \delta.
\end{eqnarray*} 
\end{lemm}
\begin{proof}

We use the estimate of Lemma \ref{decreasing} twice with $V=S_U(s_0)$ and $U_-=U$. The first time we take  $U_+=S_U(s)$, and the second time  $U_+=S_{U}(s_0)$.
The difference of the two results gives the first equality. Hypotheses $H1(a)$ and $H1(b)$ shows that the right hand side of the equality is nonpositive.  
The function $\sigma'_U$ and $\frac{d}{ds}\eta(U|S_U(s))$ are both continuous and non zero at $s=s_0$. Therefore there  exists  $0<\delta<s_0$ such that for $|s-s_0|\leq \delta$ we have both
\begin{eqnarray*}
&&|\sigma'_U(s)-\sigma'_U(s_0)|\leq |\sigma'(s_0)|/2,\\
&&\left|\frac{d}{ds}\eta(U|S_U(s))-\frac{d}{ds}\eta(U|S_U(s_0))\right|\leq \frac{1}{2}\left | \frac{d}{ds}\eta(U|S_U(s_0))\right|.
\end{eqnarray*}
And so, for $|s-s_0|\leq\delta$, we have 
\begin{eqnarray*}
&& F(S_U(s),S_U(s_0))-\sigma_U(s)\eta(S_U(s) \midd S_U(s_0))\leq -4\kappa_1|\sigma'_{U}(s_0)|^2 |s-s_0|^2\\
&&\qquad\qquad\qquad\qquad\leq - \kappa_1 |\sigma_U(s)-\sigma_U(s_0)|^2,
\end{eqnarray*}
with 
$$
\kappa_1=\frac{1}{32|\sigma'_U(s_0)|^2}|\sigma'_U(s_0)|\frac{d}{ds}\eta(U|S_U(s_0)).
$$
 Let us denote 
 $$
 \kappa_2=\inf(\eta(U|S_U(s_0))-\eta(U|S_U(s_0-\delta)); \eta(U|S_U(s_0+\delta))-\eta(U|S_U(s_0))). 
  $$
  Using that $\eta(U|S_U(s))$ is decreasing in $s$, and $\sigma'_U(s)$ is negative, we get for $s\leq s_0-\delta$
  $$
   F(S_U(s),S_U(s_0))-\sigma_U(s)\eta(S_U(s) \midd S_U(s_0))\leq -\kappa_2 \int_s^{s_0-\delta}\sigma_U'(\tau)\,d\tau=-\kappa_2|\sigma_U(s)-\sigma_U(s_0)|.
  $$ 
  in the same way we find for $s\geq s_0+\delta$
   $$
   F(S_U(s),S_U(s_0))-\sigma_U(s)\eta(S_U(s) \midd S_U(s_0))\leq -\kappa_2|\sigma_U(s)-\sigma_U(s_0)|.
  $$ 
Taking $\kappa=\inf(\kappa_1,\kappa_2)$ gives the result.
\end{proof}
The next result uses the decrease of entropy of the 1-shock family. We now consider a fixed shock $(U_L,U_R)$ with velocity $\sigma$. We denote $B(U,C)$ the ball centered at $U$ of radius $C$.
\begin{lemm}\label{lemm_def_C0}
There exist $C_0>0$,  $\beta>0$, and $v\in (\sigma,\lambda_-(U_L))$,  such that for any $U\in B(U_L,C_0)\subset B$:
\begin{eqnarray*}
&&v<\lambda_-(U),\\
&& -F(U, U_L)+v \eta(U\midd U_L)\leq -\beta \eta(U\midd U_L),\\
&& F(U, U_R)-v \eta(U\midd U_R)\leq -\beta \eta(U\midd U_R).
\end{eqnarray*}
\end{lemm}
\begin{proof}
We use Lemma \ref{cornerstone} with $U_R=S_{U_L}(s_0)$, and $s=0$. So $S_{U_L}(0)=U_L$ and (from Hypothesis  $H2$) $\sigma_{U_L}(0)=\lambda_-(U_L)$.  This gives
$$
F(U_L,U_R)-\lambda_-(U_L) \eta(U_L|U_R)<0.
$$
Since the inequality is strict, we can find $v$ with $\sigma<v<\lambda_-(U_L)$ such that we still have
$$
F(U_L,U_R)  -v \eta(U_L|U_R)<0,
$$
which can be written $-2\beta_1\eta(U_L|U_R)$, for $\beta_1$ small enough. Using the continuity of $F(\cdot,U_R)$, $\eta(\cdot|U_R)$, and $\lambda_{-}(\cdot)$ on $\mV$, there exists $C_{0,1}$ small enough such that 
$$
F(U,U_R)  -v \eta(U|U_R)<-\beta_1 \eta(U|U_R),\qquad \mathrm {and} \qquad v\leq \lambda_-(U)
$$
for  $U\in B(U_L,C_{0,1})$.
\vskip0.3cm
Doing an expansion at  $U= U_L$, we find
$$
 -F(U, U_L)+v \eta(U\midd U_L)=(U-U_L)^TD^2\eta(U_L)(v I-\nabla A(U_L))(U-U_L)+O(|U-U_L|^3).
$$ 
Since $\eta$ is a strictly convex entropy in $B$, $D^2\eta(U_L)$ is symmetric and strictly positive  and  the matrix $D^2\eta(U_L)(v I-\nabla A(U_L))$ is symmetric. Therefore those two matrices are  diagonalizable in the same basis. This gives
$$
D^2\eta(U_L)(v I-\nabla A(U_L))\leq (v-\lambda_-(U_L))D^2\eta(U_L),
$$ 
where $v-\lambda_-(U_L)=-2\beta_2<0$ thanks to Hypothesis (H2). Hence 
\begin{eqnarray*}
&&-F(U, U_L)+v\eta(U\midd U_L)\leq -2\beta_2(U-U_L)^TD^2\eta(U_L)(U-U_L)+O(|U-U_L|^3)\\
&&\qquad =-2\beta_2\eta(U\midd U_L)+O(|U-U_L|^3)\\
&&\qquad \leq -\beta_2\eta(U\midd U_L),\qquad \mathrm{for \ \ } U\in B(U_L,C_{0,2}),
\end{eqnarray*}
for $C_{0,2}$ small enough.
 Finally, taking $\beta=\inf(\beta_1,\beta_2)$, and $C_0=\inf(C_{0,1}, C_{0,2})$ gives the result.
\end{proof}
We are now ready to define $a$ which defines the metric of the contraction. Note that its definition does not depend on $\mU_K$ (and so, not on the weak solution $U(t,x)$).
We remind the reader that 
$$
\mathcal{O}_a=\{U\in \mU \setminus \eta(U|U_L)-a\eta(U|U_R)<0\}.
$$
\begin{prop}\label{lemm_a}
There exists $a^*>0$,and $0<\eps<1/2$ such that for any $0<a<a^*$, $\mathcal{O}_a\subset B(U_L,\eps C_0)$, and for every $U_-\in B(U_L,\eps C_0)$ and every $s\geq 0$ such that $\sigma_{U-}(s)\leq v$
$$
-F(U_-,U_L)+\sigma_{U_-}(s)\eta(U_-|U_L)+a\left(F(S_{U_-}(s)|U_R)-\sigma_{U_-}(s)\eta(S_{U_-}(s)|U_R)\right)\leq0.
$$
\end{prop}
Note that the inequality holds for any $s>0$, that is, for any 1-shock with $U_L$ in $B(U_L,\eps C_0)$, whatever the strength of the shock, whenever the velocity of the shock is smaller than $v$ defined in Lemma \ref{lemm_def_C0}.

\begin{proof} We study, in a first part, the set $\mathcal{O}_a$. We show the inequality in a second part.
\vskip0.3cm \noindent{\bf Step 1: Study of $\mathcal{O}_a$. }
Note that for $a<1$
$$
\eta(U|U_L)-a\eta(U|U_R)<0
$$
is equivalent to 
\begin{equation}\label{eq_a}
\eta(U)\leq \frac{1}{1-a}\left( \eta(U_L)-a\eta(U_R)-\eta'(U_L)U_L+a\eta'(U_R)U_R+[\eta'(U_L)-a\eta'(U_R)]U \right).
\end{equation}
The right-hand side of the inequality is linear in $U$. The convexity of $\eta$ implies the convexity of $\mathcal{O}_a$. Moreover, 
(\ref{eq_a}) can be rewritten as 
\begin{eqnarray*}
&&\eta(U|U_L)\leq \frac{a}{1-a}\left( \eta(U_L)-\eta(U_R)-\eta'(U_L)U_L+\eta'(U_R)U_R+[\eta'(U_L)-\eta'(U_R)]U \right)\\
&&\qquad \leq Ca(1+|U|), \qquad \mathrm{for \ \ } 0<a<1/2.
\end{eqnarray*}
Using Lemma \ref{L2} with $\Omega=B$, we find that for any $U\in B\cap\mathcal{O}_a$:
$$
C_1|U-U_L|^2\leq Ca(1+|U|)\leq C^*a.
$$
So, for $a^*$ small enough, for any $a<a^*$, we have for any $U\in B\cap \mathcal{O}_a$
$$
|U-U_L|^2\leq C^*a\leq \frac{1}{4} (\mathrm{diam} \ B)^2.
$$
The set $\mathcal{O}_a$ is convex, and $B\cap\mathcal{O}_a $ is strictly including in $B$, so 
$$
\mathcal{O}_a=\mathcal{O}_a\cap B,
$$
and for any $\eps>0$, there exists $a>0$ small enough such that 
$$
\mathcal{O}_a\subset B(U_L, \eps C_0).
$$
\vskip0.3cm \noindent{\bf Step 2: Perturbation of Lemma \ref{cornerstone}. } 
In this part, we show that for any $U\in B$, $s\geq0$, and $s_0\geq0$, we have
\begin{equation}\label{lemme4+}
\begin{array}{l}
\ds{\qquad F(S_U(s),S_U(s_0))-\sigma_U(s)\eta(S_U(s)|S_U(s_0))}\\
\ds{\qquad -\left(F(S_U(s),U_R)-\sigma_U(s)\eta(S_U(s)|U_R)\right)}\\
\ds{=F(U_R,S_U(s_0))-\sigma_U(s)\eta(U_R|S_U(s_0))}\\
\ds{\qquad+\left[\eta'(U_R)-\eta'(S_U(s_0))\right]\left[A(U)-A(U_L)-\sigma_U(s)(U-U_L)+(\sigma-\sigma_U(s))(U_L-U_R)\right],}
\end{array}
\end{equation}
where $U_R=S_{U_L}(s_0)$, and $\sigma=\sigma_{U_L}(s_0)$.
\vskip0.3cm
This equality can be computed as follows. Using the definitions of $F$ and of the relative entropy, the left hand side of (\ref{lemme4+}) can be written as
\begin{eqnarray*}
&& \qquad G(U_R)-G(S_U(s_0))-\eta'(S_U(s_0))[A(S_U(s))-A(S_U(s_0))]+\eta'(U_R)[A(S_U(s))-A(U_R)]\\
&&\qquad -\sigma_U(s)[\eta(U_R)-\eta(s_U(s_0))]+\sigma_U(s)\eta'(S_U(s_0))[S_U(s)-S_U(s_0)]\\
&&\qquad\qquad\qquad-\sigma_U(s)\eta'(U_R)[S_U(s)-U_R]\\
&&=F(U_R,S_U(s_0))+\left[\eta'(U_R)-\eta'(S_U(s_0))\right]\left[A(S_U(s))-A(U_R)\right]\\
&&\qquad -\sigma_U(s)\eta(U_R|S_U(s_0))-\sigma_U(s)\left[\eta'(U_R)-\eta'(S_U(s_0))\right][S_U(s)-U_R]\\
&&=F(U_R,S_U(s_0))-\sigma_U(s)\eta(U_R|S_U(s_0))\\
&&\qquad  +\left[\eta'(U_R)-\eta'(S_U(s_0))\right] \left[ A(S_U(s))-A(U_R)-\sigma_U(s)(S_U(s)-U_R)\right].
\end{eqnarray*}
This gives (\ref{lemme4+}) thanks to the Rankine-Hugoniot conditions
\begin{eqnarray*}
&&A(U_R)-A(U_L)=\sigma (U_R-U_L)\\
&&A(S_U(s))-A(U)=\sigma_U(s) (S_U(s)-U). 
\end{eqnarray*}
\vskip0.3cm \noindent{\bf Step 3: Control of the right-hand side of (\ref{lemme4+}). } 
In this step, we show that the right-hand side of  (\ref{lemme4+}) can be bounded by
\begin{equation}\label{estimatelemme4+}
C|U-U_L|^2 (1+|\sigma_U(s)-\sigma_U(s_0)|)+C|U-U_L|\ |\sigma_U(s)-\sigma_U(s_0)| 
\end{equation}
uniformly with respect to $s>0$ and $U\in B$, for a fixed constant $C$ depending only on the shock $(U_L, U_R, \sigma)$,  the Lipschitz norms of $A, \eta, G$ on $B\cup \tilde{B}$, where $\tilde{B}$ is the image of $B$ through $S_\cdot(s_0)$, and the Lipschitz norms on $B$ of $U\to \sigma_U(s_0)$, and $U\to S_U(s_0)$.
\vskip0.3cm
First $U\to \sigma_U(s_0)$ is bounded in $B$.
Since $S_U(s_0)$ is bounded in $\tilde{B}$ for $U\in B$, we have
\begin{eqnarray*}
&&|F(U_R,S_U(s_0))|\leq C|U_R-S_U(s_0)|^2\\
&&|\sigma_U(s_0)\eta(U_R|S_U(s_0))|\leq C|U_R-S_U(s_0)|^2\\
&&\left|\eta'(U_R)-\eta'(S_U(s_0))\right| \left|A(U)-A(U_L)-\sigma_U(s_0)(U-U_L)+[\sigma-\sigma_U(s_0)](U_L-U_R)\right|\\
&&\qquad \leq C|U_R-S_U(s_0)|\ (|U-U_L|+|\sigma-\sigma_U(s_0)|).
\end{eqnarray*}
Since $U_R=S_{U_L}(s_0)$, and $U\to S_U(s_0)$ is Lipschitz on $B$
$$
|U_R-S_U(s_0)|\leq C |U-U_L|,\qquad |\sigma-\sigma_U(s_0)|\leq C |U-U_L|.
$$
Finally, writing 
$$
\sigma_U(s)=\sigma_U(s_0)+(\sigma_U(s)-\sigma_U(s_0)),
$$
and using again that $U\to \sigma_U(s_0)$ is bounded on $B$, we get (\ref{estimatelemme4+}).
\vskip0.3cm \noindent{\bf Step 4: Proof of the inequality of the lemma. }
Using(\ref{lemme4+}) and  (\ref{estimatelemme4+}), we find
\begin{eqnarray*}
F(S_U(s),U_R)-\sigma_U(s)\eta(S_U(s)|U_R)-[F(S_U(s),S_U(s_0))-\sigma_U(s)\eta(S_U(s)|S_U(s_0))]\\
\qquad \qquad \leq C|U-U_L|^2 (1+|\sigma_U(s)-\sigma_U(s_0)|)+C|U-U_L|\ |\sigma_U(s)-\sigma_U(s_0)|.
\end{eqnarray*}
Thanks to  Lemma \ref{cornerstone}, this gives for $|s-s_0|\leq \delta$
\begin{eqnarray*}
&&F(S_U(s),U_R)-\sigma_U(s)\eta(S_U(s)|U_R)\leq -\kappa  |\sigma_U(s)-\sigma_U(s_0)|^2\\
&&\qquad\qquad\qquad+ C|U-U_L|^2 (1+|\sigma_U(s)-\sigma_U(s_0)|)+C|U-U_L|\ |\sigma_U(s)-\sigma_U(s_0)|\\
&&\qquad \leq \tilde{C}_\kappa( |U-U_L|^2+|U-U_L|^4)\leq C_\kappa|U-U_L|^2\qquad \mathrm{for} \ \ U\in B.
\end{eqnarray*}
For $|s-s_0|\geq \delta$, Lemma \ref{cornerstone} gives
\begin{eqnarray*}
&&F(S_U(s),U_R)-\sigma_U(s)\eta(S_U(s)|U_R)\leq -\kappa  |\sigma_U(s)-\sigma_U(s_0)|\\
&&\qquad\qquad\qquad+ C|U_L-U|^2 (1+|\sigma_U(s)-\sigma_U(s_0)|)+C|U-U_L|\ |\sigma_U(s)-\sigma_U(s_0)|\\
&&\qquad \leq C_\kappa |U-U_L|^2,
\end{eqnarray*}
for $U\in B(U_L,\eps C_0)$ whenever $ C( \eps C_0+|\eps C_0|^2)\leq \kappa$, which is fulfilled for $\eps$ small enough. 
Take $a^*$ such that $C_\kappa a^*\leq \beta$ and $\mathcal{O}_a\in B(U_L,\eps C_0)$. Then, thanks to Lemma \ref{lemm_def_C0} and the fact that $\sigma_U(s)\leq v$, for any $U\in B(U_L,\eps C_0)$
$$
-F(U,U_L)+\sigma_{U}(s)\eta(U|U_L)+a\left(F(S_{U}(s)|U_R)-\sigma_{U}(s)\eta(S_{U}(s)|U_R)\right)\leq0.
$$
\end{proof}

\section{Construction of the drift}
Throughout this section, we assume that $(U_L,U_R)\in\mV^2$ is a fixed 1-discontinuity with velocity $\sigma$, and that $U$ is a fixed weak entropic solution of (\ref{system}) verifying the strong trace property of Definition \ref{defi_trace}. We assume that for almost every $(t,x)$, $U(t,x)\in \mU_K$, where $\mU_K$ verifies (\ref{1*}).
We fix, once for all,  $v$ and $C_0$ as in Lemma \ref{lemm_def_C0}, and   $a>0$, $\eps>0$ verifying Proposition \ref{lemm_a}.
First, we consider the function
$$
 V(U)=v -\frac{[-F(U,U_L)+v\eta(U\midd U_L)]_++a[F(U,U_R)-v\eta(U\midd U_R)]_+}{\eta(U \midd U_L)-a\eta(U \midd U_R)}, \qquad\mathrm{for} \ \  U\in \mV.
$$
%a smooth function defined on $\mV$, verifying: 
%\begin{eqnarray*}
%&&V(U)= v, \mathrm{for \ \ U\in \Oa},\\
%&& V(U)=v -\frac{[-F(U,U_L)+v\eta(U\midd U_L)]_++a[F(U,U_R)-v\eta(U\midd U_R)]_+}{\eta(U \midd U_L)-a\eta(U \midd U_R)} \  \mathrm{for \ \ U\in (B(U_L,C_0))^c},\\
%&&V(U)\leq v\qquad \mathrm{for \ \ all \ \ }U\in \mV.
%\end{eqnarray*}
The function $U\to V(U)$ is well defined on $\mV$ thanks to Lemma \ref{lemm_def_C0}. Indeed, the numerator is equal to 0 on $B(U_L,C_0)$ which strictly contains the set $\{U \setminus \eta(U|U_L)-a\eta(U|U_R)=0\}$ where the denominator vanishes.
Note that $U\to V(U)$ can be continuously extended  on  $\mU_K$ (since it verifies (\ref{1*})).

In this section, we construct the drift $t\to x(t)$ and study its properties. We build $x(t)$, following \cite{LV}  (see Filippov \cite{Filippov}).

For any Lipschitzian path  $t\to x(t)$ we define
\begin{eqnarray*}
\Vs(t)&=&\max\left\{V(U(t,x(t)-)), \thinspace V(U(t,x(t)+))\right\},\\[0.2 cm]
\Vi(t)&=&
\begin{cases}
\min\left\{V(U(t,x(t)-)), \thinspace V(U(t,x(t)+))\right\}.
\end{cases}
\end{eqnarray*}
%where 
%\begin{eqnarray*}
%V^i_\pm(t)&=&V(U(t,x(t)\pm)),\\
%V^s_\pm(t)&=&V(U(t,x(t)\pm)),\qquad  \mathrm{if}\  U(t,x(t)\pm)\neq U_L,\\
%&=&\lambda^+_{U_L},\qquad  \mathrm{if}\  U(t,x(t)\pm)= U_L.
%\end{eqnarray*}

This section is dedicated to the following proposition.
\begin{prop}\label{x(t)}
For any $(U_L,U_R)\in\mV^2$    1-discontinuity with velocity $\sigma$,  
and $U$ a weak entropic solution of (\ref{system}) verifying the strong trace property of Definition \ref{defi_trace}, there exists a Lipschitzian path $t\to x(t)$
such that  for almost every $t>0$
$$
\Vi(t)\leq x'(t)\leq \Vs(t).
$$ 
\end{prop}
\begin{proof}
%Consider the $C^\infty$ function in $x$
Consider the function
$$
v_n(t,x)=\int_{0}^1V(U(t,x+\frac{y}{n}))\,dy.
$$
Note that, $v_n$ is bounded, measurable in $t$, and Lipschitz in $x$. We denote by $x_n$ the unique solution to 
%Since $U\to V(U)$ is bounded on $\mU$, $v_n$ is  Lipschitz in $x$.
%We denote $x_n$ the unique solution to 
\begin{equation*}
\left\{\begin{array}{l}
x_n'(t)=v_n(t,x_n(t)), \qquad t>0,\\[0.1cm]
x_n(0)=0,
\end{array}\right.
\end{equation*}
in the sense of Carath\'{e}odory.
Since $v_n$ is uniformly bounded, $x_n$ is uniformly Lipschitzian (in time) with respect to $n$. Hence, there exists a Lipschitzian path $t\to x(t)$ such that (up to a subsequence) $x_n$ converges to  $x$ in $C^0(0,T)$ for every $T>0$.
We construct $\Vs(t)$ and $\Vi(t)$ as above for this particular fixed path $t\to x(t)$. Let us show that for almost every $t>0$
\begin{align}
&\ds{\lim_{n\to\infty} \ [x'_n(t)-\Vs(t)]_+ =0,} \label{est1}\\[0.1cm]
&\ds{\lim_{n\to\infty} \ [\Vi(t) - x'_n(t) ]_+=0.} \label{est2}
\end{align}
Both limits can be proved the same way. Let us focus on the first one. We have
\begin{align*}
[x'_n(t)-\Vs(t)]_+ &= \left[\int_0^1V(U(t,x_n(t)+\frac{y}{n}))\,dy-\Vs(t)\right]_+\\[0.1 cm]
&=\left[\int_0^1[V(U(t,x_n(t)+\frac{y}{n}))-\Vs(t)]\,dy\right]_+\\[0.1 cm]
&\leq \int_0^1\left[V(U(t,x_n(t)+\frac{y}{n}))-\Vs(t)\right]_+\,dy\\[0.1 cm]
&\leq \ \esssup_{y \in (0,1)} \ \left[V(U(t,x_n(t)+\frac{y}{n}))-\Vs(t)\right]_+ \\[0.1 cm]
&\leq \esssup_{z \in (-\eps_n, \eps_n)} \ \left[V(U(t,x(t)+z))-\Vs(t)\right]_+,
\end{align*}
where, for a given $t>0$, $\eps_n \to 0$ is chosen so that $x_n(t)- x(t) \in (-\eps_n, \eps_n - \frac{1}{n})$. 
We claim that for almost every $t>0$, the last term above goes to zero as $n \to \infty$. Indeed, fix $t >0$ for which $U(t, x(t) + \cdot)$ has a left and right trace in the sense of Definition \ref{defi_trace}. That is,
\begin{align*}
\lim_{\eps \to 0} \left\{ \esssup_{y \in (0,\eps)} \ds \vert U(t, x(t) + y) - U_+(t) \vert  \right\} = \lim_{\eps \to 0} \left\{ \esssup_{y \in (0, \eps)} \ds \vert U(t, x(t) - y) - U_-(t) \vert  \right\}= 0,
\end{align*}
Since $V$ is continuous on $\mU_K$, we have that for all $r >0$ there exists $\delta > 0$ such that 
%for $y$ in the respective ranges above% for $y >0$ 
%\begin{align}
%\vert U(t,x(t) + y) - U_+(t) \vert < \delta \quad \Rightarrow  \quad  [ V(U(t,x(t) + y)) - V(U_+(t)) ]_+ < \eps.
%\end{align}
\begin{align*}
\vert U - U_\pm(t) \vert < \delta \quad \Rightarrow  \quad  [ V(U) - V(U_\pm(t)) ]_+ < r.
\end{align*}
Therefore,
\begin{align*}
\lim_{\eps \to 0} \left\{ \esssup_{y \in (0,\eps)} \ds \left[ V(U(t, x(t) \pm y)) - V(U_\pm(t)) \right]_+  \right\} = 0,
\end{align*}
and it follows easily that 
\begin{align*}
\lim_{\eps \to 0} \left\{ \esssup_{z \in (-\eps,\eps)} \ds \left[ \thinspace V(U(t, x(t) + z)) - \Vs(t) \thinspace \right]_+  \right\} = 0.
\end{align*}
This verifies the claim above and finishes the proof of (\ref{est1}). The proof of (\ref{est2}) is similar.

\vskip0.3cm
Finally, the sequence $x_n'$ converges to $x'$ in the sense of distributions. Also, the function $[\cdot]_+$  is convex. 
Therefore, integrating (\ref{est1}) and (\ref{est2}) on $[0,T]$ and passing to the limit, we obtain
%Hence at the limit:
%By virtue of (\ref{est1}) and (\ref{est2}), we obtain at the limit:
$$
\int_0^T[\Vi(t)-x'(t)]_+\,dt=\int_0^T[x'(t)-\Vs(t)]_+\,dt=0.
$$
In particular, for almost every $t>0$ we have
$$
\Vi(t)\leq x'(t)\leq \Vs(t).
$$
\end{proof}

We end this section with an elegant formulation of the Rankine-Hugoniot condition and related entropy estimates, as originally presented by Dafermos in the $BV$ case. Those estimates remain true for solutions having the strong trace property (in fact, the strong trace property defined in \cite{Vasseur_trace} suffices). The proof can be found in \cite{LV}.
%We end this section with the following lemma which shows that for almost every time $t>0$, $(U(t,x(t)-),U(t,x(t)+))$ is %an entropic Rankine-Hugoniot 
%discontinuity with velocity $x'(t)$. This was proved before in the $BV$ case by Dafermos. The proof is similar with %the regulated trace property. We give the proof in the appendix.
\begin{lemm}\label{Dafermos}
Consider $t\to x(t)$ a Lipschitzian path, and $U$ an entropic weak solution to (\ref{system}) verifying the strong trace property. Then, for almost every
$t>0$ we have
\begin{eqnarray*}
&&A(U(t,x(t)+))-A(U(t,x(t)-))=x'(t)(U(t,x(t)+)-U(t,x(t)-)),\\[0.1 cm]
&&G(U(t,x(t)+))-G(U(t,x(t)-))\leq x'(t)(\eta(U(t,x(t)+))-\eta(U(t,x(t)-))).
\end{eqnarray*}
Moreover, for almost every $t>0$ and $V\in\mV$
\begin{eqnarray*}
&&\frac{d}{dt}\int_{-\infty}^0\eta(U(t,y+x(t)) \midd V)\,dy\leq-F(U(t,x(t)-),V)+x'(t)\eta(U(t,x(t)-) \midd V),\\[0.2 cm]
&&\frac{d}{dt}\int_0^{\infty}\eta(U(t,y+x(t)) \midd V)\,dy\leq F(U(t,x(t)+),V)-x'(t)\eta(U(t,x(t)+) \midd V).
\end{eqnarray*}
\end{lemm}

\section{Proof of Theorem \ref{main}}

This section is dedicated to the proof of our main result, Theorem \ref{main}. 
%We consider a system of conservation laws (\ref{system}) such that $A$ is $C^2$  on  a open convex subset  $\mV$ of $\R^m$. We assume that the system is hyperbolic on $\mV$. That is, for any $U\in \mV$, the $m\times m$ matrix $\nabla A(U)$ is diagonalizable. 
%We denote by $\lambda^-(U)$ and $\lambda^+(U)$ the smallest and largest eigenvalues, respectively, of $\nabla A(U)$. Hereafter, we assume that  $\lambda^\pm(U)$ are simple eigenvalues for all $U\in \mV$ (But hey may be undefined on $\mU^0$, and  we do not make such hypothesis for the other eigenvalues).  
%We assume that there exists  a strictly convex entropy $\eta$ of class $C^2$ on $\mV$ verifying (\ref{entropy flux}). 
%Let $\mU_K$ 
%We assume that $\eta$, $A$ and $G$ are continuous on $\mU=\overline{\mV}$. Let $(U_L,U_R)\in \mV^2$ be an entropic 1-discontinuity. For the neighborhood $B \owns U_L$ given by (H1) we define 
% (the  radius of $B$ for the pseudo distance).

Consider $U$ weak entropic solution of (\ref{system}) with values in $\mU_K$ on $(0,T)$ verifying the strong trace property of Definition \ref{defi_trace}.
 Consider the path $t\to x(t)$ constructed in Proposition \ref{x(t)}. 
Let $a$ be such that $a<a^*$ defined in Proposition \ref{lemm_a}. We define 
$$
E_a(t)=\int_{-\infty}^{x(t)}\eta(U(t,x)|U_L)\,dx+a\int_{x(t)}^{\infty}\eta(U(t,x)|U_R)\,dx.
$$
For almost every time $t>0$, we have, from Lemma \ref{Dafermos}
\begin{eqnarray*}
&&\frac{dE_a(t)}{dt}\leq-F(U(t,x(t)-),U_L)+x'(t)\eta(U(t,x(t)-) \midd U_L)\\
&&\qquad\qquad\qquad+a\left( F(U(t,x(t)+),U_R)-x'(t)\eta(U(t,x(t)+) \midd U_R) \right).
\end{eqnarray*}
We want to show that this quantity is nonpositive for almost every time $t$. 

\vskip0.3cm 
The first result of Lemma  \ref{Dafermos} ensures that,  for almost every time $t>0$, $(U(t,x(t)-),U(t,x(t)+))$ is an admissible discontinuity with velocity $x'(t)$. So, thanks to Lemma \ref{decreasing},  for both $U_{\pm}=U(t,x(t)-)$ or $U(t,x(t)+)$ we have 
\begin{equation}\label{+-}
\frac{dE_a(t)}{dt}\leq-F(U_\pm),U_L)+x'(t)\eta(U_\pm \midd U_L)+a\left( F(U_\pm),U_R)-x'(t)\eta(U_\pm \midd U_R) \right).
\end{equation}
We denote $U_*\in \{U_-,U_+\}$ such that 
$$
V(U_*)=\max(V(U_-), V(U_+)).
$$
From Proposition \ref{x(t)} and the definition of $V$
\begin{equation}\label{2*}
x'(t)\leq V(U_*)\leq v.
\end{equation}

We consider different cases, whether $U_-=U(t,x(t)-)$ and $U_+=U(t,x(t)+)$ verify both $U_+\in \mathcal{O}^c_a$ and  $U_-\in \mathcal{O}^c_a$, or not. ($\mathcal{O}^c_a$ is the complement of $\mathcal{O}_a$ in  $ \mU_K$.)

\vskip0.3cm\noindent{\bf Step 1.} If $U_+\in \mathcal{O}^c_a$ and  $U_-\in \mathcal{O}^c_a$.
 By virtue of (\ref{+-}) we find 
\begin{eqnarray*}
&&
\frac{dE_a(t)}{dt}\leq-F(U_*,U_L)+x'(t)\eta(U_* \midd U_L)+a\left( F(U_*),U_R)-x'(t)\eta(U_* \midd U_R) \right)\\
&&\qquad\leq -F(U_*,U_L)+a F(U_*,U_R)+x'(t)[\eta(U_* \midd U_L)-a\eta(U_* \midd U_R)].
\end{eqnarray*}
Using that $\eta(U_*|U_L)-a\eta(U_*|U_R)\geq 0$ (since $U_*\in \mathcal{O}_a^c$), and  (\ref{2*}) we get: % both:
\begin{eqnarray*}
\label{EV}&&
\frac{dE_a(t)}{dt}\leq -F(U_*,U_L)+a F(U_*,U_R)+V(U_*)[\eta(U_* \midd U_L)-a\eta(U_* \midd U_R)].   %\\
%\label{Ev}&&\frac{dE_a(t)}{dt}\leq -F(U_*,U_L)+a F(U_*,U_R)+v[\eta(U_* \midd U_L)-a\eta(U_* \midd U_R)].
\end{eqnarray*}
%We consider two cases whether $U_*\in B(U_L,C_0)$ or not. If  $U_*\in  B(U_L,C_0)$, (\ref{Ev}) and  Lemma \ref{lemm_def_C0} ensures that 
%\begin{equation}\label{result}
%\frac{dE_a(t)}{dt}\leq 0.
%\end{equation}
%Now if $U_*\in  B(U_L,C_0)^c$, then using  (\ref{EV}), and 
Thanks to the definition  of $V$, we get 
\begin{eqnarray*}
&&
\frac{dE_a(t)}{dt}\leq -F(U_*,U_L)+a F(U_*,U_R)+V(U_*)[\eta(U_* \midd U_L)-a\eta(U_* \midd U_R)]\\
&& \qquad  \leq -F(U_*,U_L)+a F(U_*,U_R)+v[\eta(U_* \midd U_L)-a\eta(U_* \midd U_R)]\\
&&\qquad\qquad\qquad -[ -F(U_*,U_L)+v\eta(U_* \midd U_L)]_+-a [F(U_*,U_R)-v\eta(U_* \midd U_R)]_+\\
&&\qquad \leq0.
\end{eqnarray*}

\vskip0.3cm\noindent{\bf Step 2.} Assume that $U_-=U_+\in \mathcal{O}_a$.  From Proposition \ref{x(t)} we  have $x'(t)=V(U_-)=V(U_+)$.
The definition of $V$ gives that 
\begin{eqnarray*}
&&
\frac{dE_a(t)}{dt}\leq -F(U_-,U_L)+a F(U_+,U_R)+V(U_-)[\eta(U_- \midd U_L)-a\eta(U_+ \midd U_R)]\\
&& \qquad  = -F(U_-,U_L)+a F(U_-,U_R)+v[\eta(U_- \midd U_L)-a\eta(U_- \midd U_R)]\\
&&\qquad\qquad\qquad -[ -F(U_-,U_L)+v\eta(U_- \midd U_L)]_+-a [F(U_-,U_R)-v\eta(U_-\midd U_R)]_+\\
&&\qquad \leq0.
\end{eqnarray*}

%we have $x'(t)=v$, and so, thanks to Lemma \ref{lemm_def_C0}, we still have (\ref{result}).

\vskip0.3cm\noindent{\bf Step 3.} For the last case, we  assume that at least one of the two values $U_-$ and $U_+$ lies in $\mathcal{O}_a$, and those two values are distinct. By virtue of Lemma \ref{Dafermos}, $(U_-,U_+,x'(t))$ is a Rankine-Hugoniot discontinuity. 
We  first show that, indeed,  $U_-\in \mathcal{O}_a$ and  $(U_-,U_+,x'(t))$ is a 1-shock.
\vskip0.3cm Assume that $U_+\in \mathcal{O}_a$. Then, thanks to the Hypothesis (H2) and Lemma \ref{lemm_def_C0}, $x'(t)\geq \lambda_-(U_+)>v$. %This not possible since $x'(t)\leq v$ from the definition of $V$. 
This provides a contradiction with (\ref{2*}).
Hence, we have $U_-\in \mathcal{O}_a$. But by virtue of the definition of $V$ and Lemma  \ref{lemm_def_C0}, $x'(t)\leq v <\lambda_{-}(U_-)$. Thanks to Hypothesis (H3), this ensures that   $(U_-,U_+,x'(t))$ is a 1-shock. Proposition  \ref{lemm_a} and (\ref{2*}) ensure that we still have in this case:
%(\ref{result}) is still valid in this case.
\begin{equation}\label{result}
\frac{dE_a(t)}{dt}\leq 0.
\end{equation}

 So, we have shown  that (\ref{result}) holds true  for almost every $t>0$.
%$$
 %\frac{dE_a(t)}{dt}\leq 0.
 %$$

\vskip0.3cm Now that we have shown the contraction, we have to show the estimates on $x(t)$.
We still denote  $S(x)$ the function equal to $U_L$ for $x<0$ and $U_R$ for $x>0$.
 Let $M = \| x'(t) \|_{L^\infty}$. Then for $T>0$, we consider an even cutoff function $\phi\in C^\infty(\R)$ such that 
%$\phi(x)=1$ for $|x|\leq MT$ and such that $\phi(x)=0$ for $|x| \geq 2MT$. Further, we assume $|\phi'(x)|\leq %2(MT)^{-1}$ for all $x\in \R$.
\begin{align*}
\begin{cases}
\phi(x)=1, &\text{if $|x|\leq MT$,}\\[0.1 cm]
\phi(x)=0, &\text{if $|x|\geq 2MT$,}\\[0.1 cm]
\phi '(x)\leq 0, &\text{if $x\geq0$,}\\[0.1 cm]
|\phi'(x)|\leq 2(MT)^{-1}, &\text{if $x\in \R$.}
\end{cases}
\end{align*}
Then, for almost every $0<t< T$, we have
\begin{align*}
0&=\int_0^t\int_{\R}\phi(x)\left[\dt U+\dx A(U)\right]\,dx\,dt\\
&=\int_{\R}\phi(x)\left[\sU(x-x(t))-\sU(x)\right]\,dx-\int_0^t\int_{\R}A(\sU(x-x(s)))\phi'(x)\,dx\,ds\\
& \qquad \qquad +\int_\R\phi(x)\left[U(t,x)-\sU(x-x(t))\right]\,dx+\int_\R\phi(x)\left[\sU(x)-U^0(x)\right]\,dx\\
& \qquad \qquad \qquad \qquad -\int_0^t\int_{\R}\left[A(U(s,x))-A(\sU(x-x(s)))\right]\phi'(x)\,dx\,ds.
\end{align*}
The terms on the second line above  reduce to
$$
x(t)(U_L-U_R)-t(A(U_L)-A(U_R))=(x(t)-\sigma t) (U_L-U_R).
$$ 
The third line can be controlled by
$$
\|\phi\|_{L^2(\R)}\left(\|U(t,\cdot)-\sU(\cdot-x(t))\|_{L^2(\R)}+\|U^0-\sU\|_{L^2(\R)}\right)\leq C_K\sqrt{MT}\|U^0-\sU\|_{L^2(\R)}.
$$
Finally, since $A$ has a suitable Lipschitz property at the points $U_L, U_R \in \mV$, and is bounded on $\mU_K$,  the last term has the following bound:
\begin{multline*}
\left|\int_0^t\int_\R \left[A(U(s,x))-A(\sU(x-x(s)))\right]\phi'(x)\,dx\,ds\right|\\
\qquad \ \ \leq C_K \|\phi'\|_{L^\infty(\R)}\int_0^t\int_{-2MT}^{2MT}\left|U(s,x)-\sU(x-x(s))\right|\,dx\,ds\\
\leq \frac{2C_K}{MT} \int_0^t \sqrt{4MT} \thinspace \| U(s,\cdot)-\sU(\cdot-x(s)) \|_{L^2(\R)} \,ds \leq C_K\frac{\sqrt{T}}{\sqrt{M}}\|U^0-\sU\|_{L^2(\R)}.
\end{multline*}
Combining the estimates above we obtain for $t\leq T$
$$
|x(t)-\sigma t|\leq \frac{C_K\sqrt{T}(\sqrt{M}+1/\sqrt{M})\|U^0-\sU\|_{L^2(\R)}}{|U_L-U_R|}\leq \bar{C}_K\sqrt{T}\|U^0-\sU\|_{L^2(\R)}.
$$
This concludes the proof of the theorem.
We emphasize that while the contraction does not depend on $K$ (the $L^\infty$ size of the function $U$), the control of $x(t)$ depends on it.

\section{Proof of Proposition \ref{prop_xt}}\label{scalar}

For $r>0$, we consider the initial value
\begin{eqnarray*}
&&u^0(x)=1+\frac{\sqrt{2r\eps}}{(1-x)^{1/2+r}}\qquad \mathrm{for} \  \ x<0,\\
&&u^0(x)=-1\qquad \mathrm{for} \  \ x>0.
\end{eqnarray*}
We have 
$$
\|u^0-S\|^2_{L^2}=\eps.
$$
Note that $u^0$ is increasing for $x<0$. So, for $t>0$, $u(t,\cdot)$ is increasing on $(-\infty, x(t))$ and equal to $-1$ for $x>x(t)$.
The Rankine Hugoniot condition gives that 
$$
x'(t)=u(t,x(t)-)-1>0.
$$
The value $u(t,x(t)-)$ can be obtained by the method of characteristics:
\begin{eqnarray*}
&& u(t,x(t)-)=u^0(-y(t))\\
&&x(t)+y(t)=2tu^0(-y(t)).
\end{eqnarray*}
Note that  for any $x$ we have $u^0(x)<2$ (at least for $\eps$ small enough), so $y(t)\leq 4t$. Since $u^0$ is increasing for $x<0$, 
$$
u(t,x(t)-)=u^0(-y(t))\geq 1+\frac{\sqrt{2r\eps}}{(1+4t)^{1/2+r}}.
$$ 
Hence, 
$$
x'(t)\geq \frac{\sqrt{2r\eps}}{(1+4t)^{1/2+r}}.
$$
Integrating in time we find for $t\geq 1$
$$
x(t)\geq \frac{\sqrt{2r\eps} (1+4t)^{1/2-r}}{2(1-2r)}\geq C_r\|u^0-S\|_{L^2}t^{1/2-r}.
$$
From Leger \cite{Leger}, there exists a Lipschitz drift $t\to y(t)$ such that 
$$
\int_\R|u(t,x)-S(x-y(t))|^2\,dx
$$
is not increasing in time.  Any such $t\to y(t)$ verifies
\begin{eqnarray*}
&&
|x(t)-y(t)|=\frac{1}{2}\int_\R|S(x-x(t))-S(x-y(t))|^2\,dx\\
&&\qquad\qquad\leq  \int_\R|u(t,x)-S(x-y(t))|^2\,dx+\int_\R|u(t,x)-S(x-x(t))|^2\,dx\leq 2\eps.
\end{eqnarray*}
This ends the proof.

\noindent {\bf Acknowledgment:} The author was  partially supported by the NSF while completing this work.

\bibliography{LV.bib}

\end{document}